\documentclass{amsart}
\usepackage{amssymb} 
\usepackage{amsbsy}
\usepackage{amscd}
\usepackage{amsmath}
\usepackage{amsthm}
\usepackage{amsxtra}
\usepackage{latexsym}
\usepackage[mathscr]{eucal}
\usepackage{color}
\usepackage{upgreek}
\usepackage{wasysym}

\newcommand{\wh}{\widehat}

\newcommand{\bra}{{\langle}}
\newcommand{\ket}{{\rangle}}

\newcommand{\on}{\operatorname}
\newcommand{\+}{\mathop{\oplus}}
\renewcommand{\*}{{\otimes}}
\newcommand{\mc}{\mathcal}
\newcommand{\mf}{\mathfrak}

\newcommand{\g}{\mf{g}}
\newcommand{\h}{\mf{h}}

\newcommand{\affg}{\widehat{\mf{g}}}

\newcommand{\Z}{\mathbb{Z}}
\newcommand{\C}{\mathbb{C}}

\newcommand{\ra}{\rightarrow}

\def\leq{\leqslant}
\def\geq{\geqslant}

\DeclareMathOperator{\Spec}{Spec}
\DeclareMathOperator{\gr}{gr}

\newcommand{\nc}{\newcommand}
\nc{\la}{\lambda}
\nc{\wt}{\widetilde}
\nc{\sw}{{\mathfrak s}{\mathfrak l}}
\nc{\ghat}{\wh{\g}}
\nc{\hhat}{\wh{\h}}
\nc{\bi}{\bibitem}
\nc{\pa}{\partial}
\nc{\ppart}{(\!(t)\!)}
\nc{\pparl}{(\!(\la)\!)}
\nc{\zpart}{(\!(z)\!)}
\nc{\n}{{\mathfrak n}}
\nc{\ol}{\overline}
\nc{\bb}{{\mathfrak b}}
\nc{\su}{\wh\sw_2}
\nc{\can}{\on{can}}
\nc{\ntil}{\wt{\n}}
\nc{\pone}{{\mathbb P}^1}
\nc{\bs}{\backslash}
\nc{\al}{\alpha}
\nc{\gt}{{\mathfrak g}'}
\nc{\ds}{\displaystyle}
\nc{\Bun}{\on{Bun}}
\nc{\Gr}{\on{Gr}}

\def\neg{\negthinspace}

\nc{\ka}{\kappa}
\nc{\cka}{\check\ka}
\nc{\cmu}{{\check\mu}}

\nc{\lka}{{}^L\neg\ka}
\nc{\LP}{{}^L\neg P}
\nc{\LQ}{{}^L\neg Q}
\nc{\lkt}{{}^L\neg\wt\ka}
\nc{\OO}{\mathcal O}
\nc{\Loc}{\on{Loc}}
\nc{\sto}{\!\!\shortto\!\!}
\nc{\hsl}{\widehat{\mathfrak sl}}

\theoremstyle{theorem}
\newtheorem{Th}[subsection]{Theorem}

\newtheorem*{Ques}{Question}
\newtheorem{Pro}[subsection]{Proposition}
\newtheorem{lemma}[subsection]{Lemma}
\newtheorem{Co}[subsection]{Corollary}
\theoremstyle{remark}

\newtheorem{Rem}[subsection]{Remark}

\newtheorem{Ex}[subsection]{Example}

\numberwithin{equation}{section}

\makeatletter
\@addtoreset{equation}{section}

\makeatother

\title{A question of Joseph Ritt from the point of view of vertex algebras}
\author{Tomoyuki Arakawa}
\address{Research Institute for Mathematical Sciences, Kyoto University,
Kyoto 606-8502 JAPAN}
\email{arakawa@kurims.kyoto-u.ac.jp}
\author{Kazuya Kawasetsu}
\address{
Priority Organization for Innovation and Excellence,
Kumamoto University,
Kumamoto 860-8555 JAPAN.
}
\email{kawasetsu@kumamoto-u.ac.jp}
\author{Julien Sebag}
\address{Institut de recherche math\'ematique de Rennes\\ UMR 6625 du CNRS\\ Universit\'e de Rennes 1\\
Campus de Beaulieu\\
35042 Rennes cedex (France)}
\email{julien.sebag@univ-rennes1.fr}

\begin{document}

\begin{abstract}
Let $k$ be a field of characteristic zero.
This paper studies a problem proposed by Joseph F. Ritt
in 1950. Precisely, we prove that
\begin{enumerate}
    \item If $p\geq 2$ is an integer, for every integer $i\in\mathbb{N}$, the nilpotency index of the image of $T_i$ in the ring $k\{T\}/[T^p]$ equals $(i+1)p-i$.
    \item For every pair of integers $(i,j)$, the nilpotency index of the image of $T_iU_j$ in the ring $k\{T\}/[TU]$ equals $i+j+1$.
\end{enumerate}

 \end{abstract}

\maketitle
\section{main result}

\subsection{} Let $k$ be a field of characteristic zero. Let us denote by $k\{T\}$ (resp. $k\{T,U\}$) the differential $k$-algebra obtained by endowing the $k$-algebra $k[T_i; i\in \Z_{\geq 0}]$ (resp. $k[T_i,U_j; i,j\in \Z_{\geq 0}]$) with the differential $\partial$ defined by  $\partial T_i=T_{i+1}$ (resp. $\partial T_i=T_{i+1}$ and $\partial U_j=U_{j+1}$) for every integer $i\geq 0$ (resp. for every pair $(i,j)\in \Z_{\geq 0}
^2$ of integers).  In \cite[Appendix/5]{Ritt}, J. F. Ritt asked the following question:

\begin{Ques}
\begin{enumerate}
\item  Let $p\geq 1$ be an integer. Let $[T_0^p]$ be the differential ideal generated by $T_0^p$.
For $p>0$, $i>0$ what is the least $q(i)$ such that $T_i^{q(i)}\equiv 0 \mod{[T_0^p]}$?
\item In $k\{T,U\}$,
what is the least power $q(i,j)$ of $T_i U_j$  such that $(T_iU_j)^{q(i,j)}\equiv 0 \mod{[TU]}$?.
\end{enumerate}

\end{Ques}

\noindent For $i=1$, Ritt states in \emph{loc. cit.}, with no proof, that $q(i)=2p-1$. In \cite{OKeefe}, K. B. O'Keefe gave a proof of this formula and has shown that, for $i= 2$ and $p\geq 2$, one has  $q(i)=3p-2$. See also \cite{Mead2,Mead1,OK1} for complements or connected problems. All these works use differential algebra, and are broadly based on the reduction process  due to H. Levi (see \cite{Levi}). Despite all these results, to the best of our knowledge, Ritt's question is remained open, in general, up to 2014. In \cite{Pogudin}, G. A. Pogudin indeed provides an answer by showing that 
\begin{equation}
\label{eqn:main}
q(i)=(i+1)p-i
\end{equation}
for every integer $i\in\mathbb{Z}_{\geq 0}$. The guideline of his proof consists in injecting the differential algebra $k\{T\}/[T_0^p]$ into a Grassmann algebra, endowed with  a structure of differential algebra, and testing the vanishing of the image of $T_i^q$ in that algebra (see \cite[Lemma 2,Theorem 3]{Pogudin}). In the direction of the second question, one also deduce from the Levi reduction process the following formula (see \cite[III]{Levi}, or, e.g., see \cite{Bourqui-Sebag})
\begin{equation}
\label{eqn:second}
q(i,j)=i+j+1.
\end{equation}

\subsection{} In this article, we provide a proof of formula \eqref{eqn:main} and formula \eqref{eqn:second} using vertex algebra. In this way, our proofs deeply differ from \cite{Pogudin} and \cite{Levi,Bourqui-Sebag}, since they do not use differential algebra in the sense of Ritt. For the first formula, the key point is, up to passing over $\C$, to identify $\C\{T\}/[T_0^p]$ with
the {\em Feigin-Stoyanovsky principal subspace} 
\cite{StoFei94}
of the level $(p-1)$ vacuum representation of the affine Kac-Moody algebra $\widehat{\mf{sl}}_2$,
which is a commutative vertex algebra.
This operation allows us to  reduce formula \eqref{eqn:main} to a simple fact from representation theory. In \cite{MilPen12}, 
the Feigin-Stoyanovsky principal subspace
is 
generalized to the notion of {\em principal subalgebras} of lattice vertex algebras.
They are isomorphic to the {\em free vertex algebras} in the sense of \cite{Bor86,Roi02} (see \cite{Kaw15}).
The answer to the second question can be similarly given 
using the theory of free vertex algebras and lattice vertex operators.

\section{Vertex algebras}

\subsection{} Let $V$ be a vector space over $k$.
A {\em field} on $V$ is a formal power series $f(z)\in \mathrm{End}(V)[[z,z^{-1}]]$
such that $f(z)v\in V((z))$ for any $v\in V$.
Here, $V((z))$ is the space of formal Laurent series whose coefficients are elements of $V$.

\subsection{} A {\em vertex algebra} is a vector space $V$ equipped with
\begin{enumerate}
\item (Vacuum vector) $|0\rangle \in V$,
\item (State-field correspondence) $Y(\cdot,z):V\ra \mathrm{End}(V)[[z,z^{-1}]]$,
\item (Translation operator) $T\in \mathrm{End}(V)$,
\end{enumerate}
such that
\begin{enumerate}
\item $T|0\rangle=0$,
\item $Y(|0\rangle,z)=\mathrm{id}_V$, 
\item $Y(v,z)|0\rangle=v+O(z)$ ($v\in V$),
\item $[T, Y(v,z)]=\partial_z Y(v,z)$ ($v\in V$),
\item (locality) for any $u,v\in V$, there exists $N\in\Z$ such that
\begin{equation}\label{eqn:locality}
(z-w)^{-N}(Y(u,z)Y(v,w)-Y(v,w)Y(u,z))=0.
\end{equation}
\end{enumerate}
The biggest number $N\in \Z$ which satisfies \eqref{eqn:locality} is called the {\em
locality bound} for the pair $u,v$.
The locality bound $N$ for $u,v$ is the same as the number satisfying
\[
u(N-1)v\neq 0, \quad u(n)v=0 \quad n\geq N.
\]
Here 
we have employed the notation $Y(u,z)=\sum_{n\in\Z}u(n)z^{-n-1}$.

We note that
the multiplication
 $v\mapsto u(n)v$ is not
 associative in general.
The monomial $u_1(n_1)(u_2(n_2)(\cdots(u_m(n_m)v)\cdots))\in V$ with $u_1,\ldots,u_m,v\in V$ is simply written as
$u_1(n_1)u_2(n_2)\cdots u_m(n_m)v$.

\section{Proof of Formula \eqref{eqn:main}}

\subsection{}\label{sec:Lefschetz} Let $k'$ be a field extension of $k$. Since, for every $i\geq 0$, the polynomials $\partial^i(T_0^p)$ belong to $k\{T\}$, we observe that, for every integer $q\in\Z_{\geq 0}$, the relation $T_i^q \in [T_0^p]$ holds in $k'\{T\}$ if and only if it holds in $k\{T\}$.
Thus, we may replace $k$ with an algebraic closure of $k$, and, then, by the Lefschetz principle, assume that $k=\C$.

\subsection{} Let  $\g=\mf{sl}_2$ with the standard basis $\{e,h,f\}$,
and 
let
$\affg=\g[t,t^{-1}]\+ \C K$ be
 the 
affine Kac-Moody algebra associated with $\g=\mf{sl}_2$.
The commutation relations of $\affg$ are given by
$[x_m,y_n]=[x,y]_{m+n}+ n\delta_{m+n,0}(x|y)K$,
$[K,\affg]=0$,
where $x_m=xt^m$ and $(x|y)=\on{tr}(xy)$
for $x,y\in \g$, $m\in \Z$. Let $L_{p-1}(\g)$ be the
irreducible vacuum representation of the affine Kac-Moody algebra 
of level $p-1$,
which is 
the unique simple quotient of the
induced module
$U(\affg)\*_{U(\g[t]\+ \C K)}\C_{p-1}$,
where $\C_{p-1}$ is the one-dimensional representation of 
$\g[t]\+ \C K$ on which $\g[t]$ acts trivially and 
$K$ acts as multiplication by $p-1$.
As is well-known (\cite{FreZhu92}),
there is 
a unique vertex algebra structure on $L_{p-1}(\g)$
such that the highest weight vector $|0\ket$ is the 
vacuum vector and 
\begin{align*}
Y(x_{-1}|0\ket,z)=x(z):=\sum_{n\in \Z}x_nz^{-n-1},\quad  x\in \g.
\end{align*}


\subsection{} The \emph{Feigin-Stoyanovsky principal subspace} $W$
of $L_{p-1}(\g)$
is by definition 
 the commutative vertex subalgebra 
of  $L_{p-1}(\g)$
generated by $e(z)$. Let $\partial $ be the differential 
 of  $\C[e_{-1},e_{-2},e_{-3},\dots,]$  defined by $\partial e_{-i}=ie_{-i-1}$
 for every $i\geq 1$.
We have a surjective morphism
\begin{align}
\C[e_{-1},e_{-2},e_{-3},\dots,]\twoheadrightarrow W,\quad f\mapsto f|0\ket
\label{eq:W}
\end{align}
of differential algebras.
According to  \cite{StoFei94} (see also \cite{CalLepMil08,CalLepMil08b,Fei11,LiH}),
the kernel $J$ of the above map
is the ideal generated by the $\partial^n( e_{-1}^p)$ for all $n\in\Z_{\geq 0}$.
Therefore, we have an isomorphism  of differential algebras given by:
\begin{align}
W=\C[e_{-1},e_{-2},e_{-3},\dots,]/J \cong \C\{T\}/[T_0^p],
\quad  e_{-i}\mapsto  T_{i-1}/(i-1)!.
\label{eq:iso}
\end{align}

\begin{Rem}
Let us stress that 
the character of $W$ coincides with that of the Virasoro $(2,2p+1)$-minimal model
vertex algebra
$\on{Vir}_{2,2p+1}$
(\cite{StoFei94,Fei11}). Let us stress that, if $X=\Spec(\C[T]/\bra T^p\ket)$, the  $k$-algebra $\C\{T\}/[T_0^p]$ is isomorphic to the algebra  $\mc{O}(J_{\infty}X)$ of the arc scheme  $J_{\infty}X$  associated with $X$.
Let us then mention that the identification of
$\mc{O}(J_{\infty}X)$
with $\gr \on{Vir}_{2,2p+1}$ has been previously established in 
\cite{EkeHel}. 
\end{Rem}

\subsection{} Formula \eqref{eqn:main} immediately  follows from \eqref{eq:iso}
and Proposition \ref{Pro:observation}.

\begin{Pro}\label{Pro:observation}
We have $e_{-i}^{i(p-1)}|0\ket \ne 0$ 
and $e_{-i}^{i(p-1)+1}=0$ on $W$
for all  $i\geq 1$.
\end{Pro}
\begin{proof}
We have $[f_i,e_{-i}]=-h_0 + i K$,
and
so,
$\{f_{i},-h_0 + i K, e_{-i} \}$
forms an $\mf{sl}_2$-triple
 inside
$\affg$.
Since $L_1(\g)$ is integrable and 
$(-h_0 + i K)|0\ket =i(p-1) |0\ket$,
$|0\ket $ generates an $(i(p-1)+1)$-dimensional representation over the $\mf{sl}_2$-triple.
Therefore,
$e_{-i}^{i(p-1)}|0\ket \ne 0$ and  $e_{-i}^{i(p-1)+1}|0\ket = 0$.
\end{proof}


\section{Proof of Formula \eqref{eqn:second}}

In order to give 
a proof of second question,
we need to introduce lattice vertex algebras
and free vertex algebras.

\subsection{} As in subsection \ref{sec:Lefschetz}, we may assume that $k=\mathbb{C}$.

\subsection{} Let $L$ be an integral lattice with the $\Z$-bilinear form $(\cdot,\cdot):L\times L\rightarrow\Z$ on $L$.
We set $\h=k\otimes_\Z L$ with the extended $k$-bilinear form 
$(\cdot,\cdot):\h\times\h\rightarrow k$
and let $k[L]=\bigoplus_{\alpha\in L}k e^\alpha$ be the group algebra of $L$.
Then the vector space 
\[
V_L=M(1)\otimes k[L]
\] 
admits a natural vertex superalgebra structure, called the
{\em lattice vertex superalgebra} associated with $L$.
Here, $M(1)$ is the Heisenberg vertex algebra (Fock space)
attached to $\h$.
We have $|0\ket =1\*1$,
where we write $1$ for the vacuum vector of $M(1)$.
The state-field correspondence for $1\otimes e^\alpha$ is
\[
Y(1\otimes e^\alpha,z)=E^-(-\alpha,z)E^+(-\alpha,z)\* e_\alpha z^\alpha,
\]
where
\[
E^\pm(-\alpha,z)=\mathrm{exp}\Bigl(\sum_{n\in\pm \Z}
\frac{-\alpha(n)}{n}z^{-n}\Bigr),
\]
$z^\alpha$ is defined to be $z^{(\alpha,\beta)}$ on $M(1)\otimes e^\beta$,
and $e_\alpha :k[L]\ra k[L]$ is defined by $e_\alpha (e^\beta)=\eta(\alpha,\beta)e^{\alpha+\beta}$ with a certain cocycle $\eta$ on $L$.
The superalgebra $V_L$ is a vertex algebra if and only if 
$L$ is an even lattice.
Here $L$ is called \emph{even} if $(\alpha,\alpha)$ is even for any $\alpha\in L$.

Let $\g$ be a finite-dimensional simple Lie algebra of type ADE with 
the root lattice $Q$.
It is known that the affine vertex algebra $L_1(\g)$ of level 1 is isomorphic
to the lattice vertex algebra $V_Q$.

\subsection{}

Let $C$ be a $\Z$-basis of $L$.
The vertex subsuperalgebra $W=W(C,L)$ of $V_L$ generated by $\{e^\alpha\,;\,\alpha\in C\}$
is called a {\em principal subalgebra} of $V_L$
\cite{MilPen12}.

For instance, the Feigin-Stoyanovsky 
principal subspaces
of $L_1(\g)$ are isomorphic to the principal subalgebras $W(\Phi,Q)$ of lattice vertex algebras $V_Q$, where $\Phi$ is a base of the root system of $\g$.

\subsection{}
Let $B$ be a set and $N:B\times B\ra \Z$ a symmetric function.
The {\em free vertex superalgebra} $F=F(B,N)$ is freely generated by $B$
such that for any $a,b\in B$, the number $N(a,b)$ is the locality bound for the pair $a,b$.
It has the following universal property: any vertex superalgebra generated by $B$
satisfying
\begin{equation}\label{eqn:localfree}
(z-w)^{-N(a,b)}(Y(a,z)Y(b,w)-Y(b,w)Y(a,z))=0 \quad 
\mbox{with }a,b\in B,
\end{equation}
 is a sujective image of $F$
(see \cite{Roi02,Kaw15}).
The free vertex algebras were first mentioned in \cite{Bor86}
and constructed in \cite{Roi02}.
The construction 
in \cite{Roi02} basically proceeds as follows:
\begin{itemize}
\item Consider the free associative algebra $A$ generated by the symbols
$a(n)$ with $a\in B$ and $n\in\Z$.
\item Take an appropriate completion $\widehat{A}$ of $A$ so that 
we can
take the quotient of $\widehat{A}$ by the two-sided ideal generated by
the left-hand sides of \eqref{eqn:localfree}, which are infinite sums in general.
\item Again quotient it by the left ideal generated by $a(n)$
for $a\in B$ and $n\geq 0$, which corresponds to axiom (3) in the definition of vertex algebras. This is the free vertex algebra $F=F(B,N)$.
\item The set $B$ is embedded in $F$ by $a\mapsto a(-1)$.
\end{itemize}
See also \cite{Kaw15} for an account of Roitman's construction.
There, the completion is taken as a generalization of the {\em degreewise completion}
in the sense of \cite{MatNag10}, which is used to define the universal enveloping algebras of vertex algebras.

Let $L=L(B,N)$ be the free abelian group generated by $B$ with the $\Z$-bilinear form
$(\cdot,\cdot):L\times L\ra \Z$ defined by bilinearly extending the assignment
$(a,b)=-N(a,b)$ for $a,b\in B$.
Then the free vertex (super)algebra $F(B,N)$ is isomorphic to
the principal subalgebra $W(B,L)$ (see \cite{Roi02,Kaw15}).

\subsection{}
\label{sec:combinatorialbases}

We recall combinatorial bases of free vertex (super)algebras
from \cite{Roi02,MilPen12,Kaw15}.
Let $B$ be a set and $N:B\times B\ra \Z$ a symmetric function.
Suppose that $B$ is totally ordered with the order $<$.
The free vertex superalgebra $F=F(B,N)$ has the $\C$-basis
which consists of the monomials of the form
\[
a_m(n_m+\sum_{i=1}^{m-1}N(a_m,a_i))\cdots a_2(n_2+N(a_2,a_1))a_1(n_1)|0\rangle
\]
with $m\geq 0$, $a_1\leq a_2\leq \ldots\leq a_m\in B$ and
$n_1,n_2,\ldots,n_m\in \Z_{<0}$
such that $n_i\leq n_{i-1}$ if $a_i=a_{i-1}$ for $1<i\leq m$.

\subsection{}

Let us take over the notation of the previous subsection
and suppose from now on that $N(a,a)\in2\Z$ for every $a\in B$.
In this case, the superalgebra $F=F(B,N)$ is a vertex algebra.

Recall that a vertex algebra $V$ is called {\em commutative}
if $a(n)b=0$ for any $a,b\in V$ and $n\geq0$.
In this case, $V$ is a differential algebra
with the multiplication $a\cdot b=a(-1)b$ and differential
$\partial =T$, the translation operator of $V$.

By the construction of free vertex algebras,
we see that $F$ is commutative if and only if $N(a,b)\leq 0$.

From now on let us assume that $F$ is commutative.

Let $R_F$ be Zhu's Poisson algebra associated with $F$:
\[
R_F:=F/C_2(F),\quad C_2(F):=\mbox{span}_\C\{u(-2)v\,;\, u,v\in F\}.
\]
Let $u$ and $v$ be elements of $F$.
Note that $u(n)v\in C_2(F)$ for any $n\leq -2$.
We write by $\bar{u}\in R_F$ the image of $u$ under the canonical surjection
$F\ra R_F$.
The product on $R_F$ is defined by $\bar{u}\bar{v}=\overline{u(-1)v}$
and the Poisson bracket is $\{\bar{u},\bar{v}\}=\overline{u(0)v}$.
As we assume that $F$ is commutative, the Poisson bracket is trivial
and $R_F$ is a commutative associative algebra.

By using the combinatorial basis given in Subsection~\ref{sec:combinatorialbases},
we observe that 
\begin{equation}\label{eqn:poisson}
R_F\cong\C[B]/(ab\,;\,a,b\in B, N(a,b)\leq -1)
\end{equation}
as algebras, where $\C[B]$ is the polynomial algebra with the set $B$ of independent
variables.

\subsection{}
In this section, we describe free vertex algebras as lifts of quotients of differential algebras.
After finishing this work, we found that recent preprint \cite{LiH} proves 
Theorem \ref{sec:jetfree1} and Corollary \ref{Co:free} including super-cases
using theory of principal subspaces of lattice vertex superalgebras.
We however would like to keep the present proofs which use the theory of free vertex algebras
as it seems to be remarkably short.

Let $B$ be a set and $N:B\times B\ra \Z$ a symmetric function.
Suppose that $(a,a)\in2\Z$ for every $a\in B$.
We assume that $N(a,b)\leq 0$ for all $a,b\in B$ so that 
$F=F(B,N)$ is commutative.
Recall that $k\{B\}=k[a_i\,;\,a\in B,i\in \Z_{\geq0}]$ denotes the differential $k$-algebra with the differential $\partial$.

\begin{Th}\label{sec:jetfree1}
Let $I$ be the differential ideal generated by $\{a_{-m-1}b_0\,;\,a,b\in B,N(a,b)\leq m\leq -1\}$. Then, 
we have 
\begin{equation}\label{eqn:jetfree1}
F(B,N)\cong k\{B\}/I.
\end{equation}
\end{Th}

\begin{proof}
By the universality of $F$, we have the surjection $F\ra k\{B\}/I$
since $k\{B\}/I$ is commutative.
On the other hand, we have the surjection
$\pi:k\{B\}/I\ra W(B,L)$ from $k\{B\}/I$ to the principal subalgebra of the lattice vertex algebra $V_L$ with the lattice $L=L(B,N)$.
Since $F\cong W(B,L)$, we have the assertions.
\end{proof}

\begin{Rem}
When the lattice $L(B,N)$ is positive definite,
Theorem~\ref{sec:jetfree1} is proved in \cite[Theorem 2]{Pen14}.
\end{Rem}

\begin{Co}\label{Co:free}
The free vertex algebra $F=F(B,N)$ is isomorphic to the 
jet lift of Zhu's Poisson algebra $R_F$
\begin{equation}\label{eqn:jetfree2}
F(B,N)\cong k\{B\}/(a_0b_0\,;\,a,b\in B,N(a,b)\leq -1)
\end{equation}
if and only if $N(a,a)\in\{0,-2\}$ for every $a\in B$ and $N(a,b)\in\{0,-1\}$ for every pair $(a,b)$ with $a\neq b\in B$.
\end{Co}


\begin{proof}
Write $J$ the divisor of the right-hand side of \eqref{eqn:jetfree2}.
As $\partial(a_0a_0)=2a_1a_0$, we see that $I=J$
if $N(a,a)\in\{0,-2\}$ for any $a\in B$ and $N(a,b)\in\{0,-1\}$ for all $a\neq b\in B$.
Since $\partial^2(a_0a_0)=2a_2a_0+a_1a_1$
and $\partial(a_0b_0)=a_1b_0+a_0b_1$,
we have the second assertion.
\end{proof}

\begin{Ex}
Let $B=\{a,b\}$ with $N(a,a)=N(b,b)=0$ and $N(a,b)=N(b,a)=-1$.
Then $R_F\cong k[a,b]/(ab)$ as algebras and $F\cong k\{a,b\}/(a_0b_0)$
as differential algebras, where $F=F(B,N)$.
\end{Ex} 

\subsection{}


We now apply Corollary \ref{Co:free}
to give an answer to the second problem of Ritt.

Let $L$ be a lattice with the $\Z$-basis $C=\{\alpha,\beta\}$ and the symmetric bilinear form defined by
$$
(\alpha,\alpha)=0,\qquad(\beta,\beta)=0,\qquad (\alpha,\beta)=1.
$$
Let us consider the lattice vertex algebra $V_L$ with $\h=\C\otimes_\Z L$.
Note that we have
$$
e_\gamma.e^\delta(z)=z^{-(\gamma,\delta)}e^\delta(z).e_\gamma\qquad \gamma,\delta\in L.
$$
Note also that 
\begin{equation}\label{eqn:vacuum}
E^+(h,z)1\* e^\gamma=1\* e^\gamma
\end{equation} 
for any $h\in\h$ and $\gamma\in L$.
\begin{lemma} {\rm (cf.~\cite[Proposition 6.3.14]{LL})} \label{sec:lem1}
For any $\gamma,\delta\in L$, 
$$
E^+(\gamma,z)E^-(\delta,w)=(1-w/z)^{(\gamma,\delta)}E^-(\delta,w)E^+(\gamma,z).
$$
\end{lemma}

The vertex algebra $V_L$ is graded by 
conformal weights:
$V_L=\bigoplus_{\Delta\in \Z}(V_L)_{\Delta}$,
where 
$(V_L)_{\Delta}$  is the subspace spanned by the vectors of conformal weight $\Delta$ and
the conformal weight of $h_1(-n_1)\cdots h_m(-n_m) \otimes e^\gamma$ is
given by
$$
n_1+\cdots n_m+\frac{(\gamma,\gamma)}{2}.
$$
Note that
\begin{align}
e^{\alpha}(n)V_{\Delta}\subset V_{\Delta-n-1}.
\label{eq:weight-of-e{alpha}(n)}
\end{align}

Now, the key point is that 
the free vertex algebra $F(B,N)$ defined in Example 1 is isomorphic to the following
principal subspace $W$ of $V_L$:
$$
W=W_L(C)=\langle 1\* e^\alpha,1\* e^\beta\rangle\subset V_L.
$$
It is defined by the assignment $a\mapsto 1\* e^\alpha$ and $b\mapsto 1\* e^\beta$.

\begin{Th}\label{sec:rittab}
Let $i,j$ be non-negative integers. Then
in $W$ we have
$$
(e^\alpha(-i-1)e^\beta(-j-1))^n |0\rangle\ne 
0\quad \text{if and only if}
\quad n\leq i+j.
$$
\end{Th}
\begin{proof}
The element  $(e^\alpha(-i-1)e^\beta(-j-1))^n |0\rangle$ 
belongs  to
$M(1)\* \C e^{n(\alpha+\beta)}\subset V_L$.
Since the conformal weight of $e^{n(\alpha+\beta)}$ is $n^2$,
the conformal weight of any homogenous vector $v$ of 
$M(1)\* \C e^{n(\alpha+\beta)}$ is equal to or greater than $n^2$,
and it equals to $n^2$ if and only if $v$ coincides with  $e^{n(\alpha+\beta)}$ 
up to constant multiplication.
On the other hand,
$(e^\alpha(-i-1)e^\beta(-j-1))^n |0\rangle$ is homogenous of 
conformal weight  $(i+j)n$,
see \eqref{eq:weight-of-e{alpha}(n)}.
Hence $(e^\alpha(-i-1)e^\beta(-j-1))^n |0\rangle=0$ 
for $n>i+j$,
and
\begin{equation}\label{eqn:cond1}
(e^\alpha(-i-1)e^\beta(-j-1))^{i+j} |0\rangle = c \* e^{(i+j)(\alpha+\beta)},\qquad c\in\C.
\end{equation}
for some $c\in \C$.
It remains to show that  $c\neq 0$.
Since $W$ is commutative, the vector $(e^\alpha(-i-1)e^\beta(-j-1))^n |0\rangle$ coincides with
\begin{align}
\on{Res}
\left(z_1^{-i-1}\cdots z_{i+j}^{-i-1} w_1^{-j-1}\cdots w_{i+j}^{-j-1}
e^\alpha(z_1)\cdots e^\alpha(z_{i+j})e^\beta(w_1)\cdots e^\beta(w_{i+j}) |0\rangle\right).
\label{eq:the-value}
\end{align}
Here 
$\on{Res}f(z_1,\dots,z_{i+j},w_1,\dots, w_{i+j}) $ denotes the coefficient
of $z_1^{-1}\dots z_{i+1}^{-1}w_1^{-1}\dots w_{i+j}^{-1}$ in $f$.
Using  Lemma \ref{sec:lem1} repeatedly,
we have up to some non-zero multiple from 2-cocycle $\varepsilon$
that
\begin{align*}
&e^\alpha(z_1)\cdots e^\alpha(z_{i+j})e^\beta(w_1)\cdots e^\beta(w_{i+j})\\
&=z_1^{i+j}\cdots z_{i+j}^{i+j} \prod_{k,\ell=1}^{i+j}(1-w_k/z_\ell)
E^-(-\alpha,z_1)\cdots E^-(-\alpha,z_{i+j}) 
E^-(-\beta,w_1)\cdots E^-(-\beta,w_{i+j}) \\
&
\cdot 
E^+(-\alpha,z_1)\cdots E^+(-\alpha,z_{i+j})E^+(-\beta,w_1)\cdots E^+(-\beta,w_{i+j})\*  
e_{(i+j)\alpha}(z_1\dots z_{i+r})^{\alpha}e_{(i+j)\beta}(w_1\dots w_{i+j})^{\beta}.
\end{align*}
Hence by \eqref{eqn:vacuum},
we find that 
up to some nonzero multiple from 2-cocycle 
\eqref{eq:the-value} is
equal to
\begin{align}
&\on{Res} \left(z_1^{j-1}\cdots z_{i+j}^{j-1} w_1^{-j-1}\cdots w_{i+j}^{-j-1} \prod_{k,\ell=1}^{i+j}(1-w_k/z_\ell)
\right.\label{eq:the-value2}\\
&\quad \left.E^-(-\alpha,z_1)\cdots E^-(-\alpha,z_{i+j}) 
E^-(-\beta,w_1)\cdots E^-(-\beta,w_{i+j}) \* e^{(i+j)(\alpha+\beta)}\right).
\nonumber
\end{align}
Since 
it must be equal to 
$1\* e^{(i+j)(\alpha+\beta)}$ 
up to constant multiplication,
from weight consideration
we conclude that
\eqref{eq:the-value2} coincides with
\begin{align*}
&\on{Res} \left(z_1^{j-1}\cdots z_{i+j}^{j-1} w_1^{-j-1}\cdots w_{i+j}^{-j-1} \prod_{k,\ell=1}^{i+j}(1-w_k/z_\ell) 
1\* e^{(i+j)(\alpha+\beta)}\right).
\end{align*}
If $j=1$, then this equals $-(i+j)! \* e^{(i+j)(\alpha+\beta)}$. In a similar way, we see that
it is equal to $(-1)^j P_{i+j,j}\*  e^{(i+j)(\alpha+\beta)}$, where $P_{i+j,j}$
is the number of arrangements of 0 and 1 on the $(i+j)\times (i+j)$-square such that any row and any column have
exactly $j$ tuples of 1.
Since $P_{i+j,j}\neq 0$, we have Theorem \ref{sec:rittab}.
\end{proof}

It follows immediately
from Theorem \ref{sec:rittab}
that  the answer to the second question of Ritt is 
given by
$i+j+1$.

\subsection*{Acknowledgements}
This work was started when T.~A.~
was visiting University of Lille from May, 2019 to July, 2019.
He thanks the institute for its hospitality. 
 T.~A.~is partially supported by 
 by the Labex CEMPI (ANR-11-LABX-0007-01)
 and by JSPS KAKENHI Grant Numbers 17H01086, 17K18724.
K.~K.~is partially supported by 
MEXT Japan ``Leading Initiative for Excellent Young Researchers (LEADER)'',
JSPS Kakenhi Grant numbers 19KK0065 and 19J01093.



\bibliography{/Users/tomoyuki/Documents/Dropbox/bib/math}

\begin{thebibliography}{CLM08b}

\bibitem[B86]{Bor86}
Richard E.~Borcherds,  Vertex algebras, Kac-Moody algebras, and the Monster.
{\em Proc.~Natl.~Acad.~Sci. USA} 83.10 (1986): 3068--3071.

\bibitem[BS]{Bourqui-Sebag} David Bourqui and Julien Sebag. \newblock{Differential algebra}, in preparation

\bibitem[CLM08a]{CalLepMil08}
Corina~Calinescu, James~Lepowsky, and Antun~Milas.
\newblock Vertex-algebraic structure of the principal subspaces of certain
  {$A_1^{(1)}$}-modules. {I}. {L}evel one case.
\newblock {\em Internat. J. Math.}, 19(1):71--92, 2008.

\bibitem[CLM08b]{CalLepMil08b}
Corina~Calinescu, James~Lepowsky, and Antun~Milas.
\newblock Vertex-algebraic structure of the principal subspaces of certain
  {$A^{(1)}_1$}-modules. {II}. {H}igher-level case.
\newblock {\em J. Pure Appl. Algebra}, 212(8):1928--1950, 2008.

\bibitem[vEH]{EkeHel}
Jethro van Ekeren and Reimundo Heluani.
\newblock Chiral homology of elliptic curves and Zhu's algebra.
\newblock {arXiv:1804.00017 [math.QA]}.

\bibitem[Fei11]{Fei11}
Boris~L. Fe\u{\i}gin.
\newblock Abelianization of the {BGG} resolution of representations of the
  {V}irasoro algebra.
\newblock {\em Funktsional. Anal. i Prilozhen.}, 45(4):72--81, 2011.

\bibitem[FZ92]{FreZhu92}
Igor~B. Frenkel and Yongchang Zhu.
\newblock Vertex operator algebras associated to representations of affine and
  {V}irasoro algebras.
\newblock {\em Duke Math. J.}, 66(1):123--168, 1992.

\bibitem[Kaw15]{Kaw15}
Kazuya Kawasetsu.
The free generalized vertex algebras and generalized principal subspaces.
 {\em J. Alg.} 444 (2015): 20-51.

\bibitem[LiH]{LiH}
 Hao Li.
Some remarks on associated varieties of vertex operator superalgebras,
arXiv:2007.04522 [math-ph].



\bibitem[LL12]{LL}
 James Lepowsky and Haisheng Li. {\em Introduction to vertex operator algebras and their representations}. Vol. 227. 
 {Springer Science \& Business Media}, 2012.

\bibitem[Lev42]{Levi}
Howard~Levi.
\newblock On the structure of differential polynomials and on their theory of
  ideals.
\newblock {\em Trans. Amer. Math. Soc.}, 51:532--568, 1942.

\bibitem[MNT10]{MatNag10}
Atsushi Matsuo, Kiyokazu Nagatomo, and Akihiro Tsuchiya. Quasi-finite algebras graded by hamiltonian and vertex operator algebras.
{\em Moonshine-The First Quarter Century and Beyond: Proceedings of a Workshop on the Moonshine Conjectures and Vertex Algebras.} Vol. 372. Cambridge University Press, 2010.

\bibitem[Mea55]{Mead2}
David~G. Mead.
\newblock Differential ideals.
\newblock {\em Proc. Amer. Math. Soc.}, 6:420--432, 1955.

\bibitem[Mea73]{Mead1}
David~G. Mead.
\newblock The equation of {R}amanujan-{N}agell and {$[y^{2}]$}.
\newblock {\em Proc. Amer. Math. Soc.}, 41:333--341, 1973.

\bibitem[MP12]{MilPen12}
Antun Milas and Michael Penn. 
Lattice vertex algebras and combinatorial bases: general case and $\mathcal{W}$-algebras.
{\em New York J. Math.} 18 (2012): 621--650.

\bibitem[O'K60]{OKeefe}
Kathleen B. O'Keefe.
\newblock A property of the differential ideal {$y^{p}$}.
\newblock {\em Trans. Amer. Math. Soc.}, 94:483--497, 1960.

\bibitem[O'K66]{OK1}
Kathleen B. O'Keefe.
\newblock Unusual power products and the ideal {$[y^{2}]$}.
\newblock {\em Proc. Amer. Math. Soc.}, 17:757--758, 1966.

\bibitem[P14]{Pen14}
Michael Penn. Lattice vertex superalgebras, I: Presentation of the principal subalgebra.
{\em Comm.~Algebra} 42.3 (2014): 933-961.

\bibitem[Pog14]{Pogudin}
Gleb A. Pogudin.
\newblock Primary differential nil-algebras do exist.
\newblock {\em Moscow Univ. Math. Bull.}, 69(1):33--36, 2014.


\bibitem[Rit50]{Ritt}
Joseph F. Ritt.
\newblock {\em Differential {A}lgebra}.
\newblock American Mathematical Society Colloquium Publications, Vol. XXXIII.
  American Mathematical Society, New York, N. Y., 1950.

\bibitem[R02]{Roi02}
Michael Roitman.  Combinatorics of free vertex algebras.{\em  J.  Alg.} 255 (2002): 297-323.

\bibitem[SF94]{StoFei94}
Alexander V. Stoyanovski\u{\i} and Boris~L. Fe\u{\i}gin.
\newblock Functional models of the representations of current algebras, and
  semi-infinite {S}chubert cells.
\newblock {\em Funktsional. Anal. i Prilozhen.}, 28(1):68--90, 96, 1994.

\end{thebibliography}

\bibliographystyle{alpha}

\end{document}